\documentclass[a4paper,11pt]{article}
\usepackage{hyperref}
\usepackage{amsthm,amsmath,amssymb,amsfonts}
\usepackage[numbers, sort & compress]{natbib}

\def\pref#1{(\ref{#1})}

\makeatletter
\def\p@enumii{}
\makeatother

\newtheorem{thm}{Theorem}[section]
\newtheorem{lem}[thm]{Lemma}
\newtheorem{prop}[thm]{Proposition}
\newtheorem{cor}[thm]{Corollary}

\theoremstyle{definition}
\newtheorem{defn}[thm]{Definition}
\newtheorem{ex}[thm]{Example}
\theoremstyle{remark}

\numberwithin{equation}{section}

\newcommand{\z}{\mathbb{Z}}
\newcommand{\n}{\mathbb{N}}

\newcommand{\G}{\mathcal{G}}
\newcommand{\D}{\mathcal{D}}
\newcommand{\FF}{\mathcal{F}}
\newcommand{\FG}{\mathcal{F}'}

\newcommand{\U}{\mathrm{U}}
\renewcommand{\b}[1]{\overline{#1}}
\newcommand{\rg}{\rangle}
\renewcommand{\lg}{\langle}
\newcommand{\se}{\subseteq}

\newcommand{\sm}{\setminus }
\newcommand{\give}{$\Rightarrow$}

\newcommand{\ifof}{if and only if }

\renewcommand{\l}{\left}
\renewcommand{\r}{\right}
\renewcommand{\b}{\big}

\newcommand{\ind}{\mathrm{Ind}}
\newenvironment{psmat}
{\left(\begin{smallmatrix}}
	{\end{smallmatrix}\right)}
	
\sloppypar

\begin{document}
\title{Well-covered Unitary Cayley Graphs of Matrix Rings over Finite Fields and Applications}

\author{Shahin Rahimi$^1$ and Ashkan Nikseresht$^{2}$\\
\it\small Department of Mathematics, College of Sciences, Shiraz University, \\
\it\small 71457-44776, Shiraz, Iran\\
\it\small $^1$E-mail: shahin.rahimi.math@gmail.com\\
\it\small $^2$E-mail: ashkan\_nikseresht@yahoo.com}
\date{}

\maketitle
\begin{abstract}
Suppose that $F$ is a finite field and $R=M_n(F)$ is the ring of $n$-square matrices over $F$. Here we characterize
when the Cayley graph of the additive group of $R$ with respect to the set of invertible elements of $R$, called the
unitary Cayley graph of $R$, is well-covered. Then we apply this to characterize all finite rings with identity whose
unitary Cayley graph is well-covered or Cohen-Macaulay.
\end{abstract}


\section{Introduction}

In this paper all rings are assumed to be finite, associative and with a nonzero identity and $R$ always denotes a ring
and $F$ is a finite field. Also all graphs are simple and undirected.

Recall that the \emph{unitary Cayley graph} $\G(R)$ of $R$ is the (undirected) Cayley graph of the Abelian group
$(R,+)$ with respect to the set $\U(R)$ of unit elements of $R$. This means that the vertices of $\G(R)$ are elements
of $R$ and two vertices $x$ and $y$ are adjacent \ifof $x-y$ is invertible in $R$.  Recently, several researchers have
studied unitary Cayley graphs of rings, see for example \cite{G(R)genus,kiani,KianiLinAlg,UcayMat,yassemi,A lot,chen}.

A graph $G$ is called \emph{well-covered} when all maximal independent sets of $G$ have the same size. Well-covered
graphs are well-studied in graph theory from several view points such as algorithmic, algebraic or structural
characterization view points (see, for example, \cite{Brown, CWalker, Finbow, product, WellNew, yassemi, WellSurvey}).

In \cite{yassemi} well-covered unitary Cayley graphs of commutative rings are characterized. The main aim of this
research is to extend their results to non-commutative rings. Because of the Wedderburn-Artin Theorem on semisimple
rings, the main step of this work is to characterize when the unitary Cayley graph of the $n$-square matrix ring
$M_n(F)$ with entries in a finite field $F$ is well-covered. The answer to this question is the main result of Section
2 (see Theorem \ref{mnf}).

Recall that if $G$ is a graph on vertex set $[n]=\{1,\ldots, n\}$ and edge set $E(G)$ and $S=K[x_1,\ldots, x_n]$ where
$K$ is field, then the ideal $I(G)=\lg x_ix_j|{i,j}\in E(G) \rg$ of $S$ is called the \emph{edge ideal} of $G$. A graph
$G$ is called \emph{Cohen-Macaulay} (\emph{CM}, for short) over $K$ when the ring $S/I(G)$ is a CM ring (see, \cite{CM
ring} or \cite{hibi} for the algebraic background). In commutative algebra, edge ideals of graphs and CM graphs have
been studied extensively, see for example \cite{unitCM, tri-free, matchComp, CWalker, GorenCirc, hibi, LineGraphs,
Trung, yassemi} and the references therein.

In Section 3, we apply the aforementioned characterization to find out for which rings the graph $\G(R)$ is
well-covered or Cohen-Macaulay.

In the following lemma, we have collected two results that are essential for our work. Recall that the
\emph{independence number} $\alpha(G)$ of a graph $G$ is the cardinality of a maximum independent set of vertices and a
\emph{semisimple ring} is an Artinian ring with a zero Jacobson radical (see \cite[Theorem 15.20]{AndersonFuller}).

\begin{lem}\label{ess}
\begin{enumerate}
\item \label{ess 1} (\cite[Theorem 3.6]{kiani}) Let $F$ be a finite field and $R = M_n(F )$, where $n$ is a positive
    integer. Then $\alpha\b(\G(R)\b)=|F|^{n^2-n}$.

\item \label{ess 2} (Wedderburn–Artin Theorem, see \cite[Section 13]{AndersonFuller}) If $R$ is a semisimple ring
    then $R$ is isomorphic to $M_{n_1}(D_1) \times \cdots \times M_{n_r}(D_r)$ for some division rings $D_i$ and
    positive integers $n_i$.
\end{enumerate}
\end{lem}

\section{Well-covered unitary Cayley graphs of matrix rings over finite fields}

In this section, we characterize $n$-square matrix rings over finite fields for which the unitary Cayley graph is
well-covered. To begin with, note that $\{x_1,x_2,\ldots,x_m\}$ is an independent set of $\G(R)$, exactly whenever
$\{0,x_2-x_1,\ldots,x_m-x_1\}$ is so. Therefore, without loss of generality, we consider independent sets that contain
zero. In addition, it is not hard to see that the unitary Cayley graph of a finite field is well-covered. Our first
result shows that $\G\b(M_n(F)\b)$ is well-covered for $n\leq 2$.

\begin{prop}\label{m2f}
For any finite field $F$, if $n\leq 2$, then the graph $\G\b(M_n(F)\b)$ is well-covered.
\end{prop}
\begin{proof}
Clearly, $\G\b(M_1(F)\b)=\G(F)$ is a complete graph and well-covered. Let $A$ be an independent set of $\G\b(M_2(F)\b)$
containing zero. Also, let $|A|=m+1$ for some $0\leq m <|F|^2-1$. Note that any non-unit element of $M_2(F)$ has either
the form $\begin{psmat}
		0 \\
		r
	\end{psmat} $ or $\begin{psmat}
		r \\
		\alpha r
\end{psmat} $, for some row vector $r\in F^2$ and $\alpha \in F$. We show that $A$ cannot be a maximal independent set. To that aim,
there are three cases to investigate:
	
	Case (i): $A=\{\begin{psmat}
		0 \\
		0
	\end{psmat},\begin{psmat}
		0 \\
		r_1
	\end{psmat},\ldots,\begin{psmat}
		0 \\
		r_m
\end{psmat}\} $. Since $m<|F|^2-1$, we can find a nonzero row vector $r_{m+1}$ distinct from all $r_i$'s. Therefore, $A\cup
\{\begin{psmat}
		0 \\
		r_{m+1}
	\end{psmat}\}$ is an independent set, as claimed.
	
	Case (ii): $A=\{\begin{psmat}
		0 \\
		0
	\end{psmat},\begin{psmat}
		r_1 \\
		\alpha_1 r_1
	\end{psmat},\ldots,\begin{psmat}
		r_m \\
		\alpha_m r_m
	\end{psmat}\}$ where $\alpha_i\in F$ and $0\neq r_i \in F^2$ for all $1\leq i \leq m$. Since $\begin{psmat}
		r_1 \\
		\alpha_1 r_1
	\end{psmat}-\begin{psmat}
		r_i \\
		\alpha_i r_i
\end{psmat}$ is not invertible, we can see that $\alpha_1r_1-\alpha_ir_i= \alpha(r_1-r_i)$ for some $\alpha\in F$ or
$r_1=r_i$. Thus it can be concluded that $\alpha_i=\alpha_1$ or $r_i=\beta_i r_1$ for some $\beta_i\in F$.

Assume that $r_i$ is not a multiple of $r_1$ and $r_j=\beta_j r_1$ for some $1\leq i,j\leq m$. Then since
    $\begin{psmat}
		r_i \\
		\alpha_1r_i
	\end{psmat}-
	\begin{psmat}
		\beta_jr_1 \\
		\alpha_j\beta_jr_1
\end{psmat}$ is not invertible, we have  $\alpha_1r_i-\alpha_j\beta_jr_1=\gamma (r_i-\beta_jr_1) $ for some $\gamma\in
F$. This yields that $\alpha_1=\gamma=\alpha_j$. Therefore, either there exists an $i$ with $r_i$ not a multiple of $r_1$
and $\alpha_j=\alpha_1$ for all $j$, that is, $$A=\l\{\begin{psmat}
		0 \\
		0
	\end{psmat},\begin{psmat}
		r_1 \\
		\alpha_1 r_1
	\end{psmat},\begin{psmat}
		r_2 \\
		\alpha_1 r_2
	\end{psmat},\ldots,\begin{psmat}
		r_m \\
		\alpha_1 r_m
	\end{psmat}\r\},$$
 or $r_i$ is a multiple of $r_1$ for all $i$, that is,
	$$A=\{\begin{psmat}
		0 \\
		0
	\end{psmat},\begin{psmat}
		r_1 \\
		\alpha_1 r_1
	\end{psmat},\begin{psmat}
		\beta_2r_1 \\
		\alpha_2\beta_2 r_1
	\end{psmat},\ldots,\begin{psmat}
		\beta_mr_1 \\
		\alpha_m\beta_m r_1
	\end{psmat}\}.$$
	In the former case, since $m<|F|^2-1$, choosing $r_{m+1}\in F^2$ distinct from all $r_i$'s yields that $A\cup\{\begin{psmat}
		r_{m+1} \\
		\alpha_1 r_{m+1}
	\end{psmat}\}$
is an independent set, as desired. In the latter case, $A$ is strictly contained in the independent set $\{\begin{psmat}		\alpha r_1 \\
		\beta r_1
	\end{psmat}\mid \alpha , \beta \in F\}$, as required.
	
	Case (iii): $A=\{\begin{psmat}
		0 \\
		0
	\end{psmat},\begin{psmat}
		0 \\
		r_1
	\end{psmat},\ldots,\begin{psmat}
		0 \\
		r_s
	\end{psmat},\begin{psmat}
		r_{s+1} \\
		\alpha_{s+1}r_{s+1}
	\end{psmat},\ldots,\begin{psmat}
		r_m \\
		\alpha_mr_m
\end{psmat}\}$ where $r_i\ne 0$ for all $1\leq i \leq m$ and $\alpha_i\in F$ for $s+1\leq i \leq m$. Since
$\begin{psmat}
		r_j \\
		\alpha_j r_j
	\end{psmat}-\begin{psmat}
		0 \\
		r_k
\end{psmat}$ is not invertible for $1\leq k \leq s$ and $s+1\leq j \leq m$, we infer that $r_k$
is a multiple of $r_j$ for all $s+1\leq j \leq m$ and $1\leq k\leq s$. Hence all $r_i$ are multiples of $r_1$. Again,
$A$ is strictly contained in the independent set $\{\begin{psmat}
		\alpha r_1 \\
		\beta r_1
	\end{psmat}\mid \alpha , \beta \in F\}$, as promised.
	
Therefore, we have shown that any maximal independent set in $\G\b(M_2(F)\b)$ possesses at least $|F|^2$ elements. Now, the
result follows from Lemma \ref{ess}\pref{ess 1}.
\end{proof}

Next we prove the converse of Proposition \ref{m2f}. To this end, we introduce some new tools which we call
reduced $k$-diagonal matrices.

\begin{defn}
Let $n$ be a positive integer and $1\leq k , l\leq n$. By $D_{k,l}(a_1,\ldots,a_{n-1})$, we mean an $n\times n$ matrix,
in which, the entry $(i,j)$ is $a_{i-l}$ whenever $j-i=k$ and $i\ne l$, else it is zero (with modulo $n$ arithmetic,
where $0$ is replaced with $n$). We say that a matrix $D$ is reduced $k$-diagonal if $D=D_{k,l}(a_1,\ldots,a_{n-1})$
for some suitable $n$, $k$ and $l$. Also, for convenience, we let
$D_{k,k}(a_1,\ldots,a_{n-1})=D_{k}(a_1,\ldots,a_{n-1})$.
\end{defn}

\begin{ex}
Let $n=3$ and $D_{k,l}=D_{k,l}(a_1,\ldots,a_{n-1})$. Then we have:
$$D_{1,1}=\begin{pmatrix}
	0 & 0 & 0 \\
	0 & 0 & a_1 \\
	a_2 & 0 & 0
\end{pmatrix},
D_{1,2}=\begin{pmatrix}
	0 & a_2 & 0 \\
	0 & 0 & 0 \\
	a_1 & 0 & 0
\end{pmatrix},
D_{1,3}=\begin{pmatrix}
	0 & a_1 & 0 \\
	0 & 0 & a_2 \\
	0 & 0 & 0
\end{pmatrix},$$
$$D_{2,1}=\begin{pmatrix}
	0 & 0 & 0 \\
	a_1 & 0 & 0 \\
	0 & a_2 & 0
\end{pmatrix},
D_{2,2}=\begin{pmatrix}
	0 & 0 & a_2 \\
	0 & 0 & 0 \\
	0 & a_1 & 0
\end{pmatrix},
D_{2,3}=\begin{pmatrix}
	0 & 0 & a_1 \\
	a_2 & 0 & 0 \\
	0 & 0 & 0
\end{pmatrix},$$
$$D_{3,1}=\begin{pmatrix}
	0 & 0 & 0 \\
	0 & a_1 & 0 \\
	0 & 0 & a_2
\end{pmatrix},
D_{3,2}=\begin{pmatrix}
	a_2 & 0 & 0 \\
	0 & 0 & 0 \\
	0 & 0 & a_1
\end{pmatrix},
D_{3,3}=\begin{pmatrix}
	a_1 & 0 & 0 \\
	0 & a_2 & 0 \\
	0 & 0 & 0
\end{pmatrix}.$$	
Note that the $l$-th row in $D_{k,l}(a_1, \ldots, a_{n-1})$ is zero. Also, it can be seen that
$D_k(a_1,\ldots,a_{n-1})$ where $n=4$ is one of the followings, depending on $k$:
$$\begin{pmatrix}
	0 & 0 & 0 & 0\\
	0 & 0 & a_1 & 0\\
	0 & 0 & 0 & a_2\\
	a_3 & 0 & 0 & 0
\end{pmatrix},
\begin{pmatrix}
	0 & 0 & a_3 & 0\\
	0 & 0 & 0 & 0\\
	a_1 & 0 & 0 & 0\\
	0 & a_2 & 0 & 0
\end{pmatrix},
\begin{pmatrix}
	0 & 0 & 0 & a_2\\
	a_3 & 0 & 0 & 0\\
	0 & 0 & 0 & 0\\
	0 & 0 & a_1 & 0
\end{pmatrix},
\begin{pmatrix}
	a_1 & 0 & 0 & 0\\
	0 & a_2 & 0 & 0\\
	0 & 0 & a_3 & 0\\
	0 & 0 & 0 & 0
\end{pmatrix}.$$
\end{ex}

In the next lemma, we  construct an  independent set in $\G\b(M_n(F)\b)$ by using the reduced $k$-diagonals.

\begin{lem}\label{dk}
Let $F$ be a finite field and $n$ be a positive integer. Then $\D=\{D_k(a_1,\ldots,a_{n-1})\mid 1\leq k \leq n , a_i \in
F\}$ is an independent set of $\G\b(M_n(F)\b)$.
\end{lem}
\begin{proof}
Let $D_k=D_k(a_1,\ldots,a_{n-1})$ and $D_l=D_l(b_1,\ldots,b_{n-1})$ be two elements of $\D$. If $k=l$, it can be seen
that $D_k-D_l$ contains a zero row. Hence, assume that $k\ne l$. Observe that there are exactly two rows in $D_k-D_l$,
in which only one entry is possibly nonzero, the $k$-th and the $l$-th rows. These two entries are in the same column,
that is, the $(k+l)$-th column. Therefore, the $k$-th and the $l$-th rows are linearly dependent and $D_k-D_l$ is not
an invertible matrix in $M_n(F)$, which completes the proof.
\end{proof}

Note that the independent set $\D$ of the previous lemma contains the zero matrix as $D_k(0,\ldots, 0)$.
\begin{lem}\label{comrows}
Suppose that $m<n$ are positive integers, $F$ is a finite field, $D_k=D_k(a_1,\ldots,a_{n-1})$ is an element of $\D$ as
in Lemma \ref{dk}, and $A$ is an element of any maximal independent set of $\G\b(M_n(F)\b)$ containing $\D$. Then, the
matrix composed of any $m$ rows of $D_k$ and the other rows of  $A$  is not invertible.
\end{lem}
\begin{proof}
Denote the rows of $A$ and $D_k$ by $r_i$ and $d_i$, respectively. We use induction on $m$. First, letting $m=1$, we may
assume that the first row of the mentioned matrix is $d_1$ and $k\ne 1$. Note that $\det(A)$ and $\det\b(A-D_l(b_1,
\ldots, b_{n-1})\b)$ are zero for every $l$ and $b_1, \ldots, b_{n-1}$, by assumption. So:
	$$\begin{vmatrix}
		d_1 \\
		r_2 \\
		\vdots \\
		r_n
	\end{vmatrix}=\begin{vmatrix}
		r_1 \\
		r_2 \\
		\vdots \\
		r_n
	\end{vmatrix}-\begin{vmatrix}
		r_1-d_1 \\
		r_2 \\
		\vdots \\
		r_n
	\end{vmatrix}=\det(A)-\det\b(A-D_k(0,\ldots,a_{1-k},\ldots0)\b)=0.$$
	
Therefore, in this case we get the desired
result. Now, let $m>1$ and without loss of generality, let the $m$ rows from $D_k$ be $d_1,\ldots, d_m$. Using row
linearity of determinant, it can be said that:

		$$\begin{vmatrix}
			d_1 \\
			d_2 \\
			\vdots \\
			d_m \\
			r_{m+1} \\
			\vdots \\
			r_n
		\end{vmatrix} =\begin{vmatrix}
			r_1 \\
			d_2 \\
			\vdots \\
			d_m \\
			r_{m+1} \\
			\vdots \\
			r_n
		\end{vmatrix} - \begin{vmatrix}
			r_1-d_1 \\
			d_2 \\
			\vdots \\
			d_m \\
			r_{m+1} \\
			\vdots \\
			r_n
		\end{vmatrix}.$$
		
By the induction hypothesis the first term is zero, so:
		
			\begin{align*}
			\begin{vmatrix}
				d_1 \\
				d_2 \\
				\vdots \\
				d_m \\
				r_{m+1} \\
				\vdots \\
				r_n
			\end{vmatrix}
		 &= -\begin{vmatrix}
			r_1-d_1 \\
			d_2 \\
			\vdots \\
			d_m \\
			r_{m+1} \\
			\vdots \\
			r_n
		\end{vmatrix}
		=-\begin{vmatrix}
			r_1-d_1 \\
			r_2 \\
			\vdots \\
			d_m \\
			r_{m+1} \\
			\vdots \\
			r_n
		\end{vmatrix}+\begin{vmatrix}
			r_1-d_1 \\
			r_2-d_2 \\
			\vdots \\
			d_m \\
			r_{m+1} \\
			\vdots \\
			r_n
		\end{vmatrix}\\
		&=  -\begin{vmatrix}
			r_1\\
			r_2 \\
			\vdots \\
			d_m \\
			r_{m+1} \\
			\vdots \\
			r_n
		\end{vmatrix}+\begin{vmatrix}
			d_1 \\
			r_2 \\
			\vdots \\
			d_m \\
			r_{m+1} \\
			\vdots \\
			r_n
		\end{vmatrix}+\begin{vmatrix}
			r_1-d_1 \\
			r_2-d_2 \\
			\vdots \\
			d_m \\
			r_{m+1} \\
			\vdots \\
			r_n
		\end{vmatrix} =(-1)^2\begin{vmatrix}
			r_1-d_1 \\
			r_2-d_2 \\
			\vdots \\
			d_m \\
			r_{m+1} \\
			\vdots \\
			r_n
		\end{vmatrix},\end{align*}

	where the last equality follows from the induction hypothesis.
	Continuing this way, one can see that:
	$$\begin{vmatrix}
		d_1 \\
		d_2 \\
		\vdots \\
		d_m \\
		r_{m+1} \\
		\vdots \\
		r_n
	\end{vmatrix}=(-1)\begin{vmatrix}
		r_1-d_1 \\
		d_2 \\
		\vdots \\
		d_m \\
		r_{m+1} \\
		\vdots \\
		r_n
	\end{vmatrix}=(-1)^2\begin{vmatrix}
		r_1-d_1 \\
		r_2-d_2 \\
		\vdots \\
		d_m \\
		r_{m+1} \\
		\vdots \\
		r_n
	\end{vmatrix}=\cdots=(-1)^m\begin{vmatrix}
		r_1-d_1 \\
		r_2-d_2 \\
		\vdots \\
		r_m-d_m \\
		r_{m+1} \\
		\vdots \\
		r_n
	\end{vmatrix}.$$

Now, the obtained value is the determinant of the difference between $A$ and the matrix $B$ obtained from
$D_k(a_1,\ldots, a_{n-1})$ by replacing the $a_i$'s which correspond to the last $n-m$ rows with zero. Noting that
$B\in \D$, the assertion follows.
\end{proof}

Finally, we can present the main result of this section.

\begin{thm}\label{mnf}
	Let $F$ be a finite field. Then, the graph $\G\b(M_n(F)\b)$ is  well-covered if and only if $n\leq 2$.
\end{thm}
\begin{proof}
	Taking Proposition \ref{m2f} into account, we only have to show the implication.
Suppose $n>2$ and let $M$ be any maximal independent set of $\G\b(M_n(F)\b)$ containing $\D$ as in Lemma \ref{dk} and $A\in M$. Let $1\leq k\leq n$. The
determinant of the matrix which is obtained by replacing the $k$-th row of $D_k(1,\ldots,1)$ with the $k$-th row of $A$
is $\pm a_{k,2k}$ which must be zero by Lemma \ref{comrows} (note that here $2k$ is calculated  modulo $n$). Therefore, $M$ has the property that each of its elements has $n$ zeros in
entries $(k,2k)$ for each $1\leq k \leq n$. Note that there are $|F|^{n^2-n}$ such matrices. On the other hand, consider the invertible matrix
	$\begin{psmat}
		I_{n-2}& 0\\
		0& I^*
	\end{psmat},$ where $I_{n-2}$ is the identity matrix of order $n-2$ and $I^*=\begin{psmat}
		0& 1\\
		1& 0
\end{psmat}$. This matrix has the aforementioned property (since $n>2$) and is not in $M$, thus $|M| < |F|^{n^2-n}$. Now, Lemma \ref{ess}\pref{ess 1} allows us to
conclude the result.
\end{proof}

\section{Well-covered and Cohen-Macaulay unitary Cayley graphs of finite rings}
In this section, we utilize the results of the previous section to characterize finite rings for which the unitary Cayley graph is well-covered or
Cohen-Macaulay.

Initially, we state a crucial tool that allows us to focus on semisimple rings instead of arbitrary rings for studying
well-coveredness. In addition, this proposition generalizes \cite[Corollary 2.4]{yassemi} to all finite rings using
different methods. Here $J(R)$ denotes the Jacobson radical of $R$ (see \cite[p. 120]{AndersonFuller}).

\begin{prop}\label{r to rj}
Suppose that $R$ is a finite ring, $A$ is a subset of $R$, and $\bar{A}$ denotes the image of $A$ in $R/J(R)$. The
following statements hold.
	\begin{enumerate}
\item \label{r to rj i} $A$ is a maximal independent set of $\G(R)$ if and only if $\bar{A}$ is a maximal independent
set of $\G\b(R/J(R)\b)$ and $A=A+J(R)$.
		
\item \label{r to rj ii} $\bar{A}$ is a maximal independent set of $\G\b(R/J(R)\b)$  if and only if $A+J(R)$ is a maximal
independent set of $\G(R)$.
	\end{enumerate}
	In particular, $\G(R)$ is a well-covered graph if and only if so is $\G\b(R/J(R)\b)$.
\end{prop}
\begin{proof}
Note that if $x \in R$ and $j\in J(R)$, then $x\in\U(R)$ if and only if $x+j\in \U(R)$ \ifof $x+J(R)\in \U\b(R/J(R)\b)$.
Therefore, $A$ is an independent set of $\G(R)$ if and only if $A+J(R)$ is an independent set of $\G(R)$ if and only if
$\bar{A}$ is an independent set of $\G\b(R/J(R)\b)$. From this, \pref{r to rj i} and \pref{r to rj ii} follow.


For the last part, observe that if $A$ is a maximal independent set of $\G(R)$, then by \pref{r to rj i}, $|A|=|\bar{A}| |J(R)|$ and if
$\bar{A}$ is a maximal independent set of $\G\b(R/J(R)\b)$, then $|A+J(R)|=|\bar{A}||J(R)|$. Hence, the assertion
follows immediately.
\end{proof}

Well-coveredness of unitary Cayley graphs is settled for finite commutative rings in \cite[Theorem 2.3]{yassemi} and for
$n$-square matrix rings over finite fields in Theorem \ref{mnf}. In what follows, we also settle it for certain
decomposable finite rings from which the main result can be achieved. First, a lemma.

\begin{lem}\label{a-b}
Suppose that $F$ is a finite field and $n>1$. If $A$ is a nonzero element of $M_n(F)$, then some non-unit $B\in M_n(F)$
exists such that $A-B$ is invertible.
\end{lem}
\begin{proof}
Since $A\ne 0$, a nonzero entry exists in $A$, say $a_{ij}$. Replace the $i$-th row of $A$ with zero row to obtain
$A'$ and let $k=j-i$. Letting $B=A'-D_{k,i}(1,\ldots,1)$, we see that $A-B$ can be obtained from $D_{k,i}(1,\ldots,1)$ by replacing the
$i$-th row of $D_{k,i}$ with the $i$-th row of $A$.  So $\det(A-B)=\pm a_{ij}\ne 0$. Moreover, $B$ contains a zero row,
the
$i$-th row, which completes the proof.
\end{proof}

Recall that if $G_1$ and $G_2$ are graphs on vertex sets $V_1$ and $V_2$, then their \emph{conjunction product} is the graph
on vertex set $V_1\times V_2$, in which, two vertices $(v_1,v_2)$ and $(u_1,u_2)$ are adjacent when $v_1$ is adjacent
to $u_1$ and $v_2$ is adjacent to $u_2$ in $G_1$ and $G_2$, respectively. It is routine to check that $\G(R_1\times
R_2)$ is the conjunction product of $\G(R_1)$ and $\G(R_2)$. In \cite{product}, well-coveredness of several products of
graphs including the conjunction product is studied.

\begin{lem}[{\cite[Theorem 4 \& Proposition 11]{product}}]\label{conj prod}
Suppose that $G_1$ and $G_2$ are graphs without isolated vertices and $I_1$ and $I_2$ are maximal independent sets in
$G_1$ and $G_2$, respectively. Then $I_1\times V(G_2)$ and $V(G_1)\times I_2$ are maximal independent sets in the
conjunction product of $G_1$ and $G_2$. Also if this conjuction product is well-covered, then
$G_1$ and $G_2$ are well-covered and $\alpha(G_1)|G_2|=\alpha(G_2)|G_2|.$
\end{lem}

\begin{prop}\label{prod}
Let $F$ be a finite field and $n>1$. Then, for any finite ring $R$, the graph $\G\b(R\times M_n(F)\b)$ is not
well-covered.
\end{prop}
\begin{proof}
Let $M$ be a maximal independent set of $\G(R)$ and $X$ be the set of all non-unit elements of $M_n(F)$. We show that
$N=(R\times \{0\})\cup (M\times X)$ is a maximal independent set of $\G\b(R\times M_n(F)\b)$. Noting that units of
$R\times M_n(F)$ are of the form $(a,b)$, such that $a\in \U(R)$ and $b\in \U\b(M_n(F)\b)$, it is not difficult to see that
$N$ is an independent set. If $N\cup \{(r,A)\}$ is an independent set for some $(r,A)\in \b(R\times M_n(F)\b)\sm N$,
then $(r,A)-(r-1,0)=(1,A)$ implies that $A$ is not a unit element of $M_n(F)$ and hence $r\notin M$. Since $A\ne 0$,
utilizing Lemma \ref{a-b} we can find some non-unit $B\in M_n(F)$ such that $A-B$ is invertible.    Since $M$ is
maximal and $r\notin M$, there is an $m\in M$ such that $r-m\in \U(R)$. But then $(m,B)\in N$ and
$(r,A)-(m,B)=(r-m,A-B)$ is a unit element of $R\times M_n(F)$, that is, $(r,A)$ is adjacent to $(m,B)$,  which is in
contradiction with the independence of $N\cup\{(r,A)\}$. Consequently, $N$ is a maximal independent set.

Now, if $\G\b(R\times M_n(F)\b)$ is
well-covered, according to  Lemma \ref{conj prod}, we observe that
	$|R||F|^{n^2-n}=|M||F|^{n^2}$. Therefore, letting $|F|=q$, we have $|R|=q^n|M|$. Noting that $$|\U\b(M_n(F)\b)|=(q^n-1)(q^n-q)\cdots (q^n-q^{n-1}),$$
and that $(R\times \{0\}) \cap (M\times X)=M\times \{0\}$, it can be seen that:
	\begin{align*}
		|N|&=|R|+|M|\b(q^{n^2}-(q^n-1)\cdots(q^n-q^{n-1})-1\b)\\
		&=|M|\Big(q^{n^2}-(q^n-1)\b((q^n-q)\cdots(q^n-q^{n-1})-1\b)\Big)<|M|q^{n^2},
	\end{align*}
which is a contradiction, since by Lemma \ref{conj prod}, $M\times M_n(F)$ is a maximal independent set of
$\G\b(R\times M_n(F)\b)$. Hence the result follows.
\end{proof}

Now, we are ready to present the main results of this section.

\begin{thm}
	Let $R$ be a finite ring. Then the unitary Cayley graph of $R$ is well-covered if and only if one of the following holds:
\begin{enumerate}
\item $R/J(R)$ is isomorphic to	$F$, $F\times F$, or $M_2(F)$ for some finite field $F$;
\item $R/J(R)$ is isomorphic to $\mathbb{Z}_2^k$ for some $k\in \mathbb{N}$.
\end{enumerate}
\end{thm}
\begin{proof}
First by utilizing Proposition \ref{r to rj}, we see that $\G(R)$ is well-covered if and only if so is
$\G\b(R/J(R)\b)$. Since $R$ is an Artinian ring and $J\b(R/J(R)\b)=0$, thus $R/J(R)$ is a semisimple ring. Hence by Lemma \ref{ess}\pref{ess 2} we have $R/J(R)\cong
M_{n_1}(D_1)\times\cdots\times M_{n_r}(D_r)$ for some division rings $D_i$ and positive integers $n_i$. Also,
Wedderburn's Little Theorem states that every finite division ring is a field and therefore so are the $D_i$'s. Now, if
all the $n_i$'s are equal to $1$, then $R/J(R)$ is a commutative ring and the result follows from \cite[Theorem
2.3]{yassemi}. Hence we assume that there is some $n_i>1$. Using Proposition \ref{prod}, one can observe that if $R/J(R)$
is well-covered, then $R/J(R)\cong M_n(F)$ for some finite field $F$
and positive integer $n$. Now, by applying Theorem \ref{mnf} the assertion follows.
\end{proof}

The following example shows that there are several non-commutative rings with a well-covered unitary Cayley graph.
\begin{ex}
Let $R=T_n(F)$ be the ring of upper triangular matrices over the finite field $F$. Since $J(R)$ is the set of upper
triangular matrices which have zero on the main diagonal, we have $R/J(R)\cong F^n$. Thus $\G\b(T_n(F)\b)$ is well-covered
\ifof either $n\leq 2$ or $F=\z_2$. As another example note that $\z_m/J(\z_m)\cong \z_{m'}$, where $m'$ is the
squarefree part of $m$. Thus if we set $R=M_n(\z_m)$ for some $n\geq 2$, then $R/J(R)\cong M_n(\z_{m'})$ and is
indecomposable \ifof $\z_{m'}$ is so \ifof $m'$ is prime \ifof $m$ is a prime power. Therefore, $\G(R)$ is well-covered
\ifof $n=2$ and $m$ is a prime power. More generally, if $n\geq 2$, the ring $M_n(R)$ for an arbitrary commutative ring
$R$ is well-covered \ifof $n=2$ and $R$ is a local ring, because
$M_n(R)/J\b(M_n(R)\b)\cong M_n\b(R/J(R)\b)$ is indecomposable \ifof $R/J(R)$ is indecomposable \ifof $R/J(R)$  is a field.
\end{ex}

Next we characterize rings $R$ with  a CM unitary Cayley graph. For this, we need a necessary
combinatorial condition for Cohen-Macaulayness of graphs which is stated in terms of simplicial
complexes. Recall that a \emph{simplicial complex} on vertex set $[n]$ means a family $\Delta$ of
subsets of $[n]$ with the property that if $A\in \Delta$ and $B\se A$, then $B\in \Delta$. Each
element of $\Delta$ is called a \emph{face} of $\Delta$ and maximal elements of $\Delta$ are called
\emph{facets} of $\Delta$. Dimension of a face $A$ of $\Delta$ is $|A|-1$ and if all facets of
$\Delta$ have the same dimension, then $\Delta$ is called \emph{pure}. Also $\dim \Delta$ is
defined as the maximum dimension of all faces of $\Delta$. It is easy to see that if $G$ is a graph
on vertex set $[n]$, then the family of all independent sets of $G$, denoted by $\ind(G)$ is a
simplicial complex called the \emph{independence complex} of $G$. It is clear that $G$ is
well-covered \ifof $\ind(G)$ is pure. Now assume that $S=K[x_1,\dots, x_n]$ and for each $F\se
[n]$, let $x_F=\prod_{i\in F} x_i$ and $I_\Delta=\lg x_F|F\se [n], F\notin \Delta\rg$ be the
\emph{Stanley-Reisner ideal} of $\Delta$ in $S$. It is said that $\Delta$ is CM (over $K$), when
$S/I_\Delta$ is a CM ring. Thus it follows that a graph $G$ is CM \ifof $\ind(G)$ is CM. A pure
simplicial complex $\Delta$ is called \emph{connected in codimension 1}, when for each two facets
$F,G$ of $\Delta$, there is a sequence $F=F_0, F_1, \ldots, F_t=G$ of facets of $\Delta$, such that
$|F_{i-1} \cap F_{i}|=|F_i|-1$ for each $1\leq i\leq t$. It is well-known that if $\Delta$ is CM
over some field, then it is pure and connected in codimension 1 (see \cite[Lemma 9.1.12]{hibi}). Finally,
recall that the \emph{pure $d$-skeleton} $\Delta^{[d]}$ of $\Delta$ is the pure simplicial complex whose
facets  are all $d$-dimensional faces of $\Delta$. Now we study when independence complex of
$\G(R)$ is  connected in codimension 1.

\begin{lem}\label{st conn}
	If $n\geq 2$, $\Delta$ is the independence complex of $\G\b(M_n(F)\b)$ and $d=\dim \Delta$, then
	$\Delta^{[d]}$ is not connected in codimension 1.
\end{lem}
\begin{proof}
	Let $R=M_n(F)$. Note that by Lemma \ref{ess}\pref{ess 1}, $d=|F|^{n^2-n}-1$. Since the family of
	matrices with a fixed zero row is a maximal independent set of $\G(R)$ with size $d+1=|F|^{n^2-n}$,
	we see that $\Delta$ has more than one facet. Thus if $\Delta^{[d]}$ is connected in codimension 1,
	for each facet $A$ of $\Delta$, there must exist a facet $B$ of $\Delta$ with $|A\cap B|=d$. Below we
	show that this does not hold and hence the claim follows.
	
	Let $\FF$ be the family of all matrices with first row equal to zero and $M$ be an arbitrary
	element of $\FF$. We prove that there is no independent set containing $\FF\sm\{M\}$ other  than
	$\FF$. On the contrary, if $\FG$ is such an independent set, then there is a matrix $N\in \FG$ with
	nonzero first row $r_1\in F^n$. Extend $\{r_1\}$ to two different bases $\{r_1,r'_2, \ldots,
	r'_n\}$ and $\{r_1,r''_2,\ldots, r''_n\}$ of $F^n$. This is possible since $n\geq 2$. Suppose that the $i$-th row of $N$ is $r_i$. Let $X$ and $X'$ be matrices with zero first rows
	and $i$-th rows equal to $r_i-r'_i$ and $r_i-r''_i$, respectively. At least one of $X$ and $X'$, say
	$X$, is in $\FF\sm\{M\}\se \FG$. But then the rows of $N-X$ are linearly independent and $N$ is
	adjacent to $X$ in $\G(R)$, contradicting independence of $\FG$. From this the result follows.
\end{proof}

\begin{lem}\label{st conn J=0}
Suppose that $\Delta$ is the independence complex of  $\G(R)$ and $d=\dim \Delta$. If $J(R)\neq 0$, then
	$\Delta^{[d]}$ is not connected in codimension 1.
\end{lem}
\begin{proof}
According to Proposition \ref{r to rj}, facets of $\Delta^{[d]}$ are of unions of cosets of $J(R)$ in $R$. Thus at
least two such facets exist and if $A$ and $B$ are two such facets, $|A\cap B|$ is a multiple of $|J(R)|$. But if
$\Delta^{[d]}$ is connected in codimension 1, for each facet $A$ of $\Delta$, there must
exist a facet $B$ of $\Delta$ with $|A\cap B|=d=|A|-1$. This is not a multiple of $|J(R)|$, unless $|J(R)|=1$, as required.
\end{proof}

Now we can state when $\G(R)$ is a CM graph. We should mention that for commutative rings, this result is proved in
\cite[Theorem 3.3]{yassemi}. Also recall that a simplicial complex $\Delta$ is said to be \emph{shellable} when its
facets can be sorted as $F_1,\ldots, F_t$, such that for each $i>1$ the simplicial complex $\lg F_1, \ldots, F_{i-1}\rg
\cap \lg F_i \rg$ is pure and has dimension
$\dim F_i -1$, where $\lg F_1, \ldots, F_{s}\rg$ means the simplicial complex whose facets are exactly $F_1,\ldots, F_s$. We say a graph is shellable whenever
its independence complex is so.
\begin{thm}\label{cayleyCM}
	Suppose that $R$ is a finite ring. The following are equivalent.
\begin{enumerate}
\item \label{cayleyCM1} The unitary Cayley graph of $R$ is CM.
\item \label{cayleyCM2} The unitary Cayley graph of $R$ is well-covered and shellable.
\item \label{cayleyCM3} Either $R$ is a field or $R\cong \z_2^k$ for some $k\in \n$.
\end{enumerate}
\end{thm}
\begin{proof} \pref{cayleyCM3} \give \pref{cayleyCM2} is part of \cite[Theorem 3.3]{yassemi} (note that there, by
shellable, it is meant pure shellable). \pref{cayleyCM2} \give \pref{cayleyCM1} is well-known (see, for example,
\cite[Theorem 8.2.6]{hibi}). Now assume \pref{cayleyCM1} holds. To prove \pref{cayleyCM3},
 by Lemma \ref{st conn J=0} and \cite[Lemma 9.1.12]{hibi}, we may assume that $J(R)=0$.
	If $R$ is commutative, then \pref{cayleyCM3} holds by \cite[Theorem 3.3]{yassemi}. If $R$ is a
	non-commutative indecomposable ring, then $R\cong M_n(F)$ for some $n\geq 2$ and it follows from
	Lemmas \ref{st conn} and \ref{mnf} that $\ind\b(\G(R)\b)$ is not a pure complex connected in
	codimension 1 and hence it is not CM. Now if $R$ is non-commutative but it is decomposable, then it
	follows from Proposition \ref{prod} that $\G(R)$ is not well-covered and hence it is not CM.
\end{proof}

We say that a graph $G$ is \emph{Gorenstein} (over $K$) when the ring $S/I(G)$ is Gorenstein.
\begin{cor}\label{GorenCay}
For a finite ring $R$, the graph $\G(R)$ is Gorenstein \ifof $R\cong \z_2^k$ for some $k\in \n$.
\end{cor}
\begin{proof}
The result follows from Theorem \ref{cayleyCM}, since every Gorenstein ring is CM and because of the well-known facts that the only complete Gorenstein graphs are $K_1$ and $K_2$ and
disjoint union of Gorenstein graphs is Gorenstein.
\end{proof}

It should be mentioned that it looks there is a small inaccuracy in \cite[Theorem  3.4]{yassemi}. There it is claimed that if $R$ is a field, then $\G(R)$ is Gorenstien. But
when $R$ is a field, $\G(R)$ is a complete graph and hence it is Gorenstein \ifof $\G(R)$ is $K_2$ and hence $R\cong \z_2$.


\end{document}